\documentclass[12pt]{article}
\usepackage{amsmath,amssymb,amsthm}
\usepackage{url}
\usepackage[arrow,matrix,curve]{xy}
\newtheorem{thm}{Theorem}
\newtheorem{lem}{Lemma}

\newtheorem*{thm*}{Theorem}
\newtheorem*{prop*}{Proposition}
\newtheorem*{cor*}{Corollary}
\theoremstyle{definition}

\theoremstyle{definition}
\newtheorem{rmk}{Remark}
\theoremstyle{definition}

\begin{document}

\author{Hanno von Bodecker\footnote{Fakult{\"a}t f{\"u}r Mathematik, Ruhr-Universit{\"a}t Bochum, 44780 Bochum, Germany}}

\title{The beta family at the prime two and modular forms of level three}

\date{}

\maketitle

\begin{abstract}{We use the orientation underlying the Hirzebruch genus of level three to map the beta family at the prime $p=2$ into the ring of divided congruences. This procedure, which may be thought of as the elliptic greek letter beta construction, yields the $f$--invariants of this family.}
\end{abstract}

\section{Introduction and statement of the results}


One of the most fundamental problems in pure mathematics is to understand the structure of the stable homotopy groups of the sphere $\pi_*S^0$, and the Adams--Novikov spectral sequence (ANSS) serves as a powerful tool to attack this problem: Working locally at a fixed prime $p$, we have
$$\textrm{E}_{2}^{s,t}=\textrm{Ext}^{s,t}_{BP_*BP}\left(BP_*,BP_*\right)\Rightarrow\left(\pi_{t-s}S^0\right)_{\left(p\right)},$$
and much insight can be gained by resolving its $E_2$--term into $v_n$--periodic components \cite{Ravenel:2004xh}. In their seminal paper propagating this chromatic approach, Miller, Ravenel, and Wilson introduced the so-called greek letter map, and computed the 1--line (for all primes) and the 2--line (for odd primes), generated by the alpha and beta families, respectively \cite{Miller:1977ya}. The computation of the 2--line for $p=2$ is due to Shimomura \cite{Shimomura:1981oy}: Let us concentrate on the beta elements at $p=2$ (there are also products of  $\alpha$'s): Starting from certain elements $x_i\in v_2^{-1}BP_*$, $y_i\in v_1^{-1}BP_*$, put 
$$a_0=1,\ a_1=2,\ a_k=3\cdot2^{k-1}\ k\geq2;$$
then, for $n\geq0$, odd $s\geq1$, $j\geq1$, $i\geq0$, subject to the conditions
$$n\geq i,\ 2^i|j,\ j\leq a_{n-i},\textrm{ and }j\leq2^n\textrm{ if }s=1 {\textrm{ and }} i=0,$$
the simplest beta elements are given by
\begin{equation}\label{simple_betas}
\beta_{s\cdot2^n/j,i+1}=\eta\left({x_n^s}/{2^{i+1}v_1^j}\right),
\end{equation}
where $\eta$ is the universal greek letter map. In fact, it is sometimes possible to improve divisibility, viz.\ for $n$, $s$, $j$, and $i$ as above with the additional conditions that
$$n\geq i+1\geq2, j=2\mbox{ and }s\geq3\mbox{ if }n=2,\mbox{ and }j\leq a_{n-i-1}\mbox{ if }n\geq3,$$
Shimomura defines
\begin{equation}\label{higher_beta}
\beta_{s\cdot2^n/j,i+2}=\eta\left({x_n^s}/{2^{i+2}y_i^m}\right)\quad\mbox{where } m=j/2^i,
\end{equation}
and shows the following relations between the betas in \eqref{simple_betas} and \eqref{higher_beta}:
\begin{itemize}
\item[(i)]{$\beta_{s\cdot2^n/j,i+2}=\beta_{s\cdot2^n/j,\left(i+1\right)+1}$ if $2^{i+1}|j$,}
\item[(ii)]{$2\beta_{s\cdot2^n/j,i+2}=\beta_{s\cdot2^n/j,i+1}$}.
\end{itemize}


There are striking number-theoretical patterns lurking in the stable stems which become visible from the chromatic point of view, e.g.\ the (nowadays) well-known relation between the 1--line and the (denominators of the) Bernoulli numbers. Concerning the 2--line,  Behrens has established a precise relation between the beta family for primes $p\geq5$ and the existence of modular forms satisfying appropriate congruences \cite{Behrens:2009dz}. On the other hand, using an injection of the 2--line into the ring of divided congruences (tensored with $\mathbb{Q/Z}$), Laures introduced the $f$--invariant as a higher analog of the $e$--invariant \cite{Laures:1999sh}. Subsequent work   has shown how these approaches can be merged and used to derive the $f$--invariant of the beta family, albeit still only for $p\geq5$ \cite{Behrens:2008fp}. A  different route has been taken in \cite{Hornbostel:2007ss}, where the $f$--invariant is represented using Artin--Schreier theory; however, although no longer limted to primes $p\geq5$, the calculations actually carried out in that reference only take care of two subfamilies (viz.\ $\beta_t$ for $p\nmid t$ and $\beta_{s2^n/2^n}$).

 Since there has been some progress on our geometrical understanding of the $f$--invariant through analytical techniques (to an extent where explicit calculations can be done, cf.\ e.g.\ \cite{Bodecker:2008pi}) it is desirable to have some sort of `comparison table'; to this end, we compute the $f$--invariant of the beta family\footnote{The situation of   products of   permanent alpha elements has been studied in  \cite{Bodecker:2009kx}.} at the prime $p=2$. More precisely, we take a look at the following diagram (at the level $N=3$, i.e.\ $\Gamma=\Gamma_1(3)$):
\begin{equation}\label{the_diagram}
\xymatrix{
\textrm{Ext}^0\left(BP_*,BP_*/\left(p^{\infty},v_1^{\infty}\right)[v_2^{-1}]\right)\ar[rr]\ar[d]& &\textrm{Ext}^{2,*}\left(BP_*,BP_*\right)\ar[d]\\
\textrm{Ext}^0\left(E^{\Gamma}_*,E^{\Gamma}_*/\left(p^{\infty},v_1^{\infty}\right)\right)\ar[rr]\ar@{.>}[drr]& &\textrm{Ext}^{2,*}\left(E^{\Gamma}_*,E^{\Gamma}_*\right)\ar[d]\\
 & &{\underline{\underline{D}}_{*+1}^{\Gamma}\otimes\mathbb{Q/Z}}\\}
 \end{equation}  
The composition of the vertical arrows on the RHS (which are injective by the results of \cite{Laures:1999sh}) accounts for the algebraic portion of the $f$--invariant, while the upper horizontal arrow produces the beta family. So, in order to compute the $f$--invariant of a member of this family, we  chase its preimage through the composition of the vertical arrow on the LHS and the dotted arrow; put differently, we carry out (a sufficiently large portion of) the elliptic greek letter construction explicitly. The result can be summarized as follows (where, as usual, we abbreviate $\beta_{k/j}=\beta_{k/j,1}$ and $\beta_k=\beta_{k/1}$):


\begin{thm}\label{order-two}
The $f$--invariants of the beta elements of order two are
\begin{itemize}
\item[\textup{(i)}] for odd $s\geq3$: 
$$f\left(\beta_{s}\right)\equiv\frac{1}{2}\left(\frac{E_1^2-1}{4}\right)^{s}\mod\underline{\underline{D}}_{3s-1}^{\Gamma}$$
\item[\textup{(ii)}] for odd $s\geq1$:
 $$f\left(\beta_{2s/j}\right)\equiv\frac{1}{2}\left(\frac{E_1^2-1}{4}\right)^{2s}\mod\underline{\underline{D}}_{6s-j}^{\Gamma}$$
\item[\textup{(iii)}] for  $l\geq0$ and odd $s\geq1$:
\begin{equation*}
\begin{split}
f\left(\beta_{4s\cdot2^l/j}\right)&\equiv\frac{1}{2}\left(\frac{E_1^2-1}{4}\right)^{4s\cdot2^l}+\frac{1}{2}\left(\frac{E_1^2-1}{4}\right)^{\left(4s-1\right)2^l}\mod\underline{\underline{D}}^{\Gamma}_{12s\cdot2^l-j}\\
&\equiv\frac{1}{2}\left(\frac{E_1^2-1}{4}\right)^{4s\cdot2^l}\textrm{{\em\ if }}j\leq3\cdot2^l\\
\end{split}
\end{equation*}
\end{itemize}
\end{thm}

\begin{thm}\label{higher-order}
The $f$--invariants of the beta elements of higher order are
\begin{itemize}
\item[\textup{(i)}] for odd $s\geq1$:
$$f\left(\beta_{4s/2,2}\right)\equiv\frac{1}{4}\left(\frac{E_1^2-1}{4}\right)^{4s}\mod\underline{\underline{D}}_{12s-2}^{\Gamma}$$

\item[\textup{(ii)}] for $l\geq0$, $i\geq1$, $j=m\cdot2^i\leq a_{l+2}$, odd $s\geq1$,  and mod $\underline{\underline{D}}_{3s\cdot2^{l+i+2}-j}^{\Gamma}$: 
\begin{align*}
f\left(\beta_{s\cdot2^{l+i+2}/j,i+1}\right)&\equiv\frac{1}{2^{i+1}}\left(\frac{E_1^2-1}{4}\right)^{s\cdot2^{l+i+2}}+\frac{1}{2}\left(\frac{E_1^2-1}{4}\right)^{\left(s\cdot2^{i+2}-1\right)2^{l}}\\
&\equiv\frac{1}{2^{i+1}}\left(\frac{E_1^2-1}{4}\right)^{s\cdot2^{l+i+2}}\textrm{{\em\ if }}j\leq3\cdot2^l
\end{align*}

\item[\textup{(iii)}] for $k\geq2$:
$$f\left(\beta_{4k/2,3}\right)\equiv\frac{1+4k}{8}\left(\frac{E_1^2-1}{4}\right)^{4k}\mod\underline{\underline{D}}_{12k-2}^{\Gamma}$$

\item[\textup{(iv)}] for $l\geq0$, $i\geq1$, $j=m\cdot2^i\leq a_{l+2}$,  odd $s\geq1$, and mod $\underline{\underline{D}}_{3s\cdot2^{l+i+3}-j}^{\Gamma}$:
\begin{equation*}
\begin{split}
f\left(\beta_{s\cdot2^{l+i+3}/j,i+2}\right)&\equiv\frac{1}{2^{i+2}}\left(\frac{E_1^2-1}{4}\right)^{s\cdot2^{l+i+3}}+\frac{1}{2}\left(\frac{E_1^2-1}{4}\right)^{\left(s\cdot2^{i+3}-1\right)2^l}\\
&\equiv\frac{1}{2^{i+2}}\left(\frac{E_1^2-1}{4}\right)^{s\cdot2^{l+i+3}}\textrm{{\em\ if }}j\leq3\cdot2^l\\
\end{split}
\end{equation*}
\end{itemize}
\end{thm}


The proof presented in the following section turns out to be a pretty much straightforward calculation: After a brief recollection of the relevant definitions, we study the image (under the orientation underlying the Hirzebruch genus) of the elements $x_i$ and $y_i$ occurring in the definition of the beta elements. Then, we are going to sketch our approach to the argument given in \cite[section 4]{Behrens:2008fp}, i.e.\ we explain how to carry out the greek letter map on the level of (holomorphic) modular forms. The final step consists of performing this computation  explicitly.

\section{Proof of the Theorems}


\subsection{Preliminaries}

Following \cite{Laures:1999sh}, we consider the congruence subgroup $\Gamma=\Gamma_1(N)$ for a fixed level $N>1$, set $\mathbb{Z}^{\Gamma}=\mathbb{Z}[\zeta_N,1/N]$ and denote by $M^{\Gamma}_*$ the graded ring of modular forms w.r.t.~$\Gamma$ which expand integrally, i.e.~which lie in $\mathbb{Z}^{\Gamma}[\![q]\!]$. The ring of {\em divided congruences} $D^{\Gamma}$  consists of those rational combinations of modular forms which expand integrally; this ring can be filtered by setting
$$D_k^{\Gamma}=\left\{\left.f={\textstyle{\sum_{i=0}^{k}}}f_i\ \right| f_i\in M_i^{\Gamma}\otimes\mathbb{Q},\ f\in\mathbb{Z}^{\Gamma}[\![q]\!]\right\}.$$
Finally, we introduce $$\underline{\underline{D}}^{\Gamma}_{k}=D^{\Gamma}_k+M_0^{\Gamma}\otimes\mathbb{Q}+M_k^{\Gamma}\otimes\mathbb{Q}.$$

Now, if $Ell^{\Gamma}$ denotes the complex oriented elliptic cohomology theory associated to the universal curve over the ring of modular forms w.r.t.\ $\Gamma$, the composite
$$\textrm{E}_2^{2,2k+2}[MU]\rightarrow \textrm{E}_2^{2,2k+2}[Ell^{\Gamma}]\rightarrow\underline{\underline{D}}^{\Gamma}_{k+1}\otimes{\mathbb{Q/Z}}$$
is injective (away from primes dividing the level $N$) \cite{Laures:1999sh}. Henceforth, we fix $p=2$ and $N=3$. Thus we have
$$M^{\Gamma}_*=\mathbb{Z}^{\Gamma}[E_1,E_3],$$
where
\begin{equation*}
\begin{split}
E_1&=1+6\sum_{n=1}^{\infty}\sum_{d|n}(\frac{d}{3})\ q^n,\\
E_3&=1-9\sum_{n=1}^{\infty}\sum_{d|n}(\frac{d}{3})d^2\ q^n\\
\end{split}
\end{equation*}
are the odd Eisenstein series of the indicated weight at the level $N=3$ (and $(\frac{d}{3})$ denotes the Legendre symbol).
Furthermore, the following basic congruence can be read off of the $q$-expansions:
\begin{equation}\label{the_basic_congruence}
 E_3-1\equiv \frac{E_1^2-1}{4}\mod 2D_3^{\Gamma}.
\end{equation}


\subsection{The image under the orientation}\label{computing_the_image}

Recall that the power series associated to the Hirzebruch elliptic genus of level three may be expressed as (see e.g.\ \cite{Bodecker:2008pi})
$$Q\left(x\right)=\exp\left(3\sum_{n=1}^{\infty}\frac{x^{2n}}{(2n)!}G_{2n}^*(\tau)-2\sum_{k=0}^{\infty}\frac{x^{2k+1}}{(2k+1)!}G_{2k+1}^{(-\omega)}(\tau)\right)\in M_*^{\Gamma}\otimes{\mathbb{Q}}[\![x]\!]$$
where $\omega=2\pi i/3$ and
\begin{equation*}
\begin{split}
G_{2n}^*(\tau)&=G_{2n}(\tau)-3^{2n-1}G_{2n}(3\tau),\\
G_{2k+1}^{(-\omega)}(\tau)&=\frac{e^{\omega}-e^{-\omega}}{2}3^{2k}\frac{B_{2k+1}(1/3)}{2k+1}E_{2k+1}^{\Gamma_1(3)}(\tau).\\
\end{split}
\end{equation*}
The first few terms of this power series, when expressed in terms of $E_1$ and $E_3$, i.e.\ the generators of $M_*^{\Gamma}$, read
\begin{equation}\label{expandedgenus}
\begin{split}
Ell^{\Gamma_1(3)}(x) & = 1+\frac{iE_1}{2\sqrt{3}}x+\frac{E_1^2}{12}x^2+\frac{iE_1^3-iE_3}{18\sqrt{3}}x^3+\frac{13E_1^4-16E_1E_3}{2160}x^4\\ &\quad+\frac{iE_1^2(E_1^3-E_3)}{216\sqrt{3}}x^5+\frac{121E_1^6-152E_1^3E_3+40E_3^2}{272160}x^6\\ &\quad+\frac{iE_1}{\sqrt{3}}\frac{7E_1^6-11E_1^3E_3+4E_3^2}{19440}x^7+O(x^8)\\
\end{split}
\end{equation}
The genus of the following complex projective spaces is readily evaluated:
\begin{align*}
w_1&=\phi({\mathbb{CP}}^1)=\frac{i}{\sqrt3}E_1,\\
w_3&=\phi({\mathbb{CP}}^3)=\frac{i}{\sqrt3}\frac{5E_1^3-2E_3}{9},\\
w_7&=\phi({\mathbb{CP}}^7)=\frac{i}{\sqrt3}\frac{70E_1^4E_3-14E_1E_3^2-65E_1^7}{243}.
\end{align*}
As is well-known, underlying this genus is the complex orientation of the cohomology theory $Ell^{\Gamma}$, i.e.\ 
$$\phi:MU\rightarrow Ell^{\Gamma}$$
and we can compute the images of the Hazewinkel generators \cite[Appendix A2]{Ravenel:2004xh} at the prime $p=2$, which we still denote by $v_i$:
$$v_1=w_1=\frac{i}{\sqrt3}E_1$$
$$v_2=\frac{w_3-w_1^3}{2}=\frac{i}{\sqrt{3}}\frac{4E_1^3-E_3}{9},$$
$$v_3=\frac{w_7}{4}-\frac{w_1^7+w_1w_3^2}{8}=\frac{iE_1}{\sqrt{3}}\frac{5E_1^3E_3-E_3^2-4E_1^6}{81}.$$
In particular, we see that $v_3$ is decomposable:
\begin{equation}\label{decomp_v_3}
\begin{split}
v_3&=\frac{iE_1}{\sqrt{3}}\left(\frac{4E_1^3E_3-E_3^2}{81}-\frac{4E_1^6-E_1^3E_3}{81}\right)\\
&=\frac{iE_1}{\sqrt{3}}\left(\frac{i}{\sqrt{3}}\frac{4E_1^3-E_3}{9}\right)\left(-\frac{i}{3\sqrt{3}}\left(E_3-E_1^3\right)\right)\\
&=3v_1v_2\left(v_2+v_1^3\right)\\
\end{split}
\end{equation}
Plugging \eqref{decomp_v_3} into the definitions of the $x_i$ (considered in $v_2^{-1}M_*^{\Gamma}$) yields
\begin{equation}\label{the_x_i}
\begin{split}
x_0&=v_2\\
x_1&=v_2^2-v_1^2v_2^{-1}v_3=v_2^2-3v_1^3\left(v_2+v_1^3\right)\\
x_2&=x_1^2-v_1^3v_2^3-v_1^5v_3=v_2^4-7v_1^3v_2^3+15v_1^9v_2+9v_1^{12}\\
x_i&=x_{i-1}^2\quad i\geq3,\\
\end{split}
\end{equation}
showing that the $x_i$ are actually holomorphic. On the other hand, unless $i=0$, this is not  true for the $y_i\in v_1^{-1}M_*^{\Gamma}$, which read:
\begin{equation}\label{the_y_i}
\begin{split}
y_0&=v_1\\
y_1&=v_1^2-4v_1^{-1}v_2\\
y_i&=y_{i-1}^2\quad i\geq2.\\
\end{split}
\end{equation}
However, for $i\geq1$ and $m\geq1$, we may introduce 
\begin{equation}\label{the_z_i}
\begin{split}
z_{i,m}&=v_1^{m\cdot2^i}-m\cdot2^{i+1}v_1^{m\cdot2^i-3}v_2,\\
\end{split}
\end{equation}
which are holomorphic for $m\cdot2^i\geq4$ and satisfy
\begin{align*}
z_{i,m}&\equiv y_i^m  \mod2^{i+2}v_1^{-1}M_*^{\Gamma}\\
 &\equiv 1 \ \  \mod2^{i+2}\mathbb{Z}^{\Gamma}[\![q]\!],
\end{align*}
the second line being an immediate consequence of \eqref{the_basic_congruence}.


\subsection{Determining `elliptic' beta elements}

Requiring $p>3$ and working with the full modular group, Behrens and Laures have shown in \cite[section 4]{Behrens:2008fp} how an element in $H^0\left(M_*/\left(p^{\infty},E_{p-1}^{\infty}\right)\right)$ gives rise to an element in $D\otimes{\mathbb{Q}}/D[\frac{1}{6}]+M_k\otimes\mathbb{Q}+\mathbb{Q}$; clearly, the other primes can be treated analogously by working with a smaller congruence subgroup. Let us rephrase their argument in a language closer to the original formulation of the greek letter construction:

Still working at the prime $p=2$ and the level $N=3$, we choose a (holomorphic) modular form $\mu \in M^{\Gamma}_{|\mu|}$ and a pair of positive integers $\left(i_0, i_1\right)$ such that
 \begin{equation}
 \mu^{i_1}\equiv1\mod 2^{i_0}D_{i_1|\mu|}^{\Gamma};
 \end{equation}
in particular, this ensures that  $\left(2^{i_0},\mu^{i_1}\right)$ is regular on $M_*^{\Gamma}$.
 
 Now, given a modular form $\tilde\varphi_t\in M_t^{\Gamma}$, we can use the natural inclusion
 $$M_t^{\Gamma}\hookrightarrow D_t^{\Gamma}$$
and ask whether $\tilde\varphi_t$ satisfies
\begin{equation}\label{invariant_mod}
\tilde\varphi_t\equiv\mu^{i_1}\varphi_{t/i_1|\mu|,i_0} \mod 2^{i_0}D^{\Gamma}_t
\end{equation}
for some 
$$\varphi_{t/i_1|\mu|,i_0}\in D^{\Gamma}_{t-i_1|\mu|}/2^{i_0}D^{\Gamma}_{t-i_1|\mu|}$$
Let us call a modular form satisfying \eqref{invariant_mod} {\em invariant mod} $\left(2^{i_0},\mu^{i_1}\right)$. Moreover, we have the obvious composition
$$\underline{\underline{\left(\ \cdot\ \right)}}\ \colon D_k^{\Gamma}/2^{i_0}D^{\Gamma}_k\cong{D^{\Gamma}_k}\otimes \mathbb{Z}/2^{i_0} \rightarrow D^{\Gamma}_k\otimes \mathbb{Q/Z}  \rightarrow \underline{\underline{D}}_k^{\Gamma}\otimes \mathbb{Q/Z},$$ $$\quad\varphi_k\mapsto\underline{\underline{\varphi}}_k$$
Then it is easy to see that, for an invariant modular form $\tilde\varphi_t$, the assignment
$$\tilde\varphi_t\mapsto\underline{\underline{\varphi}}_{t/i_1|\mu|,i_0}$$
depends only on the reduction $\varphi_t\equiv\tilde\varphi_t\mod\left(2^{i_0},\mu^{i_1}\right)$, hence descends to a well-defined map
\begin{equation}\label{elliptic_beta_map}
\ker\left(M_t^{\Gamma}/\left(2^{i_0},\mu^{i_1}\right)\rightarrow D_t^{\Gamma}/\left(2^{i_0},\mu^{i_1}\right)\right)\longrightarrow \underline{\underline{D}}_{t-i_1|\mu|}^{\Gamma}\otimes \mathbb{Q/Z}
\end{equation}
which we may think of as the {\em `elliptic' greek letter beta map} and which corresponds to the dotted arrow in \eqref{the_diagram}.

\begin{rmk} By removing the constant term of the $q$-expansion, we obtain another map 
$$d\colon M_t^{\Gamma}\rightarrow D^{\Gamma}_t,\quad d\left(\tilde\varphi_t\right)=\tilde\varphi-q^0\left(\tilde\varphi_t\right)$$
that might look like a more natural choice w.r.t.\ which invariance should be defined (cf.\ \cite{Behrens:2008fp}). However, we have $q^0(\varphi)\equiv \mu^{i_1}q^0(\varphi)\mod 2^{i_0}D^{\Gamma}_t$, hence both choices are equivalent (up to the shift of $\varphi_{t/i_1|\mu|,i_0}$ by the constant $q^0\left(\tilde\varphi_t\right)$, which maps to zero in $\underline{\underline{D}}^{\Gamma}_{k}\otimes\mathbb{Q/Z}$).
\end{rmk}


\subsection{Explicit computations}


In this subsection, we are going to compute the effect of the map \eqref{elliptic_beta_map} on the preimage of Shimomura's beta elements; the ones defined by \eqref{simple_betas} are dealt with easily, since   $\left(2^{i+1},v_1^{j}\right)$ is regular on $M^{\Gamma}_*$ provided that $j=m\cdot2^i$; moreover, for $k\geq0$ this implies:
\begin{equation}\label{multi_by_v_1^j}
\left(\frac{E_1^2-1}{4}\right)^k\equiv v_1^j\left(\frac{E_1^2-1}{4}\right)^k\mod2^{i+1}D^{\Gamma}_{2k+j}
\end{equation}
Furthermore, the following two results are useful:
\begin{lem}\label{E_3_congruences}
For $i\geq0$, $l\geq0$, $m\cdot2^i=j\leq6\cdot2^l$ we have:
\begin{equation*}
\begin{split}
E_3^{s\cdot2^{l+i+2}}&\equiv\left(\frac{E_1^2-1}{4}\right)^{s\cdot2^{l+i+2}}\mod2^{i+1}D^{\Gamma}_{12s\cdot2^{l+i}}+v_1^j\cdot M_{12s\cdot2^{l+i}-j}^{\Gamma}\\
\end{split}
\end{equation*}
\end{lem}
\begin{proof}
It is easy to see that for $l\geq0$ and $i\geq0$, we have
\begin{equation}\label{dunno}
E_3^{2^{l+i+2}}\equiv\left(E_3-v_1^3\right)^{2^{l+i+2}}+2^{i+1}\left(v_1^6E_3^2\right)^{2^l}E_3^{2^{l+2}\left(2^i-1\right)}\mod\left(2^{i+2},v_1^{12\cdot2^l}\right),
\end{equation}
and the basic congruence \eqref{the_basic_congruence} implies
\begin{equation}
\left(E_3-v_1^3\right)^{2^k}\equiv\left(\frac{E_1^2-1}{4}\right)^{2^k}\mod2^{k+1}D^{\Gamma}_{3\cdot2^k}\qedhere
\end{equation}
\end{proof}

\begin{lem}\label{removal_of_second_summand}
 For $i\geq0$, $l\geq0$, $1\leq j\leq6\cdot2^l$ we have:
\begin{align*}
E_3^{\left(s\cdot2^{i+2}-1\right)2^l}&\equiv\left(\frac{E_1^2-1}{4}\right)^{\left(s\cdot2^{i+2}-1\right)2^l}& &\mod2D^{\Gamma}_{12s\cdot2^{l+i}}+v_1^j\cdot M^{\Gamma}_{12s\cdot2^{l+i}-j}\\
&\equiv0& &\textrm{\quad{\em if }}j\leq3\cdot2^l
\end{align*}
\end{lem}
\begin{proof} This immediately follows from \eqref{the_basic_congruence}\end{proof}


\noindent Now let us treat the beta elements of order two, i.e.\  those with $i=0$ in \eqref{simple_betas}:

\begin{proof}[\bf{Proof of Theorem \ref{order-two}:}]\ \\
For part (i), we observe that:
\begin{align*}
x_0^s&=v_2^s\\
&\equiv E_3^s& &\mod 2D^{\Gamma}_{3s}\\
&\equiv \left(E_3-E_1^3\right)^s & &\mod2D^{\Gamma}_{3s}+v_1\cdot M^{\Gamma}_{3s-1}\\
&\equiv \left(\frac{E_1^2-1}{4}\right)^s & &\mod2D^{\Gamma}_{3s}+v_1\cdot M^{\Gamma}_{3s-1}
\end{align*}
Similarly, for part (ii) we have:
\begin{align*}
x_1^s&\equiv v_2^s & &\mod v_1^j\\
&\equiv E_3^{2s}& &\mod 2D^{\Gamma}_{6s}+v_1^j\cdot M^{\Gamma}_{6s-j}\\
&\equiv \left(E_3-E_1^3\right)^{2s}& &\mod2D^{\Gamma}_{6s}+v_1^j\cdot M^{\Gamma}_{6s-j}\\
&\equiv \left(\frac{E_1^2-1}{4}\right)^{2s}& &\mod2D^{\Gamma}_{6s}+v_1^j\cdot M^{\Gamma}_{6s-j}
\end{align*}
and since $j\leq a_{l+2}=6\cdot2^l$ (and $j\leq 2^{l+2}$ if $s=1$), for part (iii) we conclude: 
\begin{align*}
x_{2+l}^s&\equiv v_2^{4s\cdot2^l}+v_1^{3\cdot2^l}v_2^{\left(4s-1\right)2^l}& &\mod\left(2,v_1^{a_{l+2}}\right)\\
&\equiv E_3^{4s\cdot2^l}+E_3^{\left(4s-1\right)2^l}& &\mod2D^{\Gamma}_{12s\cdot2^l}+v_1^j\cdot M^{\Gamma}_{12s\cdot2^l-j}\\
&\equiv \left(\frac{E_1^2-1}{4}\right)^{4s\cdot2^l}+\left(\frac{E_1^2-1}{4}\right)^{\left(4s-1\right)2^l}& &\mod2D^{\Gamma}_{12s\cdot2^l}+v_1^j\cdot M^{\Gamma}_{12s\cdot2^l-j}
\end{align*}
In view of \eqref{multi_by_v_1^j}, this completes the proof.
\end{proof}

\begin{rmk}
Since $x_0=v_2$ is sent to zero under the map \eqref{elliptic_beta_map} w.r.t.\ $\left(2,v_1\right)$, we see that in order to obtain something interesting, we have to impose $s\geq3$ in part (i). In a similar vein, the condition $j\leq 2^{l+2}$ if $s=1$ in part (iii) is needed to ensure that  $D^{\Gamma}_{8s\cdot2^l+j}\subset D^{\Gamma}_{12s\cdot2^l}$ when using \eqref{multi_by_v_1^j}.
\end{rmk}


Next, we turn our attention to the elements  $\beta_{4s\cdot2^l/j,i+1}$ for $i\geq1$:

\begin{lem}\label{reduction_of_x_n}
For $l\geq0$ and $i\geq0$, we have 
\begin{equation*}
x_{l+i+3}\equiv v_2^{2^{l+i+3}}+2^{i+1}v_1^{3\cdot2^{l}}v_2^{\left(2^{i+3}-1\right)2^{l}}\mod\left(2^{i+2},v_1^{a_{l+2}}\right)
\end{equation*}
\end{lem}

\begin{proof} Since $(a+b)^{2^{l+1}}\equiv a^{2^{l+1}}+b^{2^{l+1}}+2(ab)^{2^{l}}\mod4$ for $l\geq0$, we compute
\begin{equation*}
x_{l+3}=x_2^{2^{l+1}}\equiv v_2^{8\cdot2^l}+2\left(v_1^{3}v_2\right)^{2^{l}}v_2^{6\cdot2^l}\mod\left(4,v_1^{a_{l+2}}\right)
\end{equation*}
and use the binomial theorem.
\end{proof}


\begin{proof}[{\bf Proof of Theorem \ref{higher-order}, part (i):}]\ \\
The choice $n=2$ and $i=1$ in \eqref{simple_betas} dictates $j=2$, hence we compute
\begin{align*}
x_2^s&\equiv v_2^{4s}& &\mod \left(4,v_1^2\right)\\
&\equiv E_3^{4s} & &\mod 4D^{\Gamma}_{12s}+v_1^2\cdot M_{12s-2}^{\Gamma}\\
&\equiv \left(\frac{E_1^2-1}{4}\right)^{4s}& &\mod 4D^{\Gamma}_{12s}+v_1^2\cdot M_{12s-2}^{\Gamma}
\end{align*}
Combined with \eqref{multi_by_v_1^j}, this yields the claim.
\end{proof}


\begin{proof}[{\bf Proof of Theorem \ref{higher-order}, part (ii):}]\ \\
In order to treat the remaining cases of our computation of $x_n^s$ mod $\left(2^{i+1},v_1^j\right)$, we notice that since \eqref{simple_betas} requires $j=m\cdot2^i\leq a_{n-i}$,  and since all cases with $i=0$ and the case $i=1$ for $n=2$ have already been taken care of, it suffices to consider $n=l+i+2$ where $l\geq0$ and $i\geq1$; now,  for odd $s\geq1$ we have (by Lemma \ref{reduction_of_x_n} in a reindexed form)
\begin{align*}
x_{l+i+2}^s&\equiv v_2^{s\cdot2^{l+i+2}}+2^iv_1^{3\cdot2^l}v_2^{s\cdot2^{l+i+2}-2^l}& &\mod\left(2^{i+1},v_1^{a_{l+2}}\right)\\
&\equiv E_3^{s\cdot2^{l+i+2}}+2^iE_3^{s\cdot2^{l+i+2}-2^l}& &\mod2^{i+1}D^{\Gamma}_{12s\cdot2^{l+i}}+v_1^j\cdot M_{12s\cdot2^{l+i}-j}^{\Gamma}
\end{align*}
from which the desired result follows.
\end{proof}


Finally, we treat the beta elements defined by \eqref{higher_beta}:

\begin{proof}[{\bf Proof of Theorem \ref{higher-order}, part (iii):}]\ \\
In order to compute the $f$--invariant of $\beta_{4k/2,3}$, we are going to show that, although $z_{1,1}=v_1^2-4v_1^{-1}v_2$ is  not holomorphic, we can still make sense out of the map \eqref{elliptic_beta_map} w.r.t.\ $\left(8, z_{1,1}\right)$ if $t=12k\geq24$. To this end, we observe  $$v_1^6=z_{1,1}v_1^4+4v_1^3v_2=z_{1,1}\left(v_1^4+4v_1v_2\right)+16v_2^2,$$
hence we compute
\begin{align*}
x_2^k&\equiv v_2^{4k}+kv_1^3v_2^{4k-1}& &\mod\left(8,v_1^6\right)\\
&\equiv \left(1+4k\right)v_2^{4k}& &\mod\left(8,z_{1,1}\right)\\
&\equiv \left(1+4k\right)E_3^{4k}& &\mod 8D^{\Gamma}_{12k}+z_{1,1}M^{\Gamma}_{12k-2}
\end{align*}
where $z_{1,1} M_{12k-2}^{\Gamma}\subset M_{12k}^{\Gamma}$ for dimensional reasons; finally, we note that
\begin{align*}
E_3^{4k}&\equiv\left(\frac{E_1^2-1}{4}\right)^{4k}& &\mod 8D^{\Gamma}_{12k}+z_{1,1}M^{\Gamma}_{12k-2}\\
&\equiv \left(\frac{E_1^2-1}{4}\right)^{4k}v_1^4z_{1,1} & &\mod 8D^{\Gamma}_{12k}+z_{1,1}M^{\Gamma}_{12k-2}\quad{\textrm{if\ }}k\geq2 
\end{align*}
which completes the proof.\end{proof}


\begin{proof}[{\bf Proof of Theorem \ref{higher-order}, part (iv):}]\ \\
Recall that in the definition \eqref{higher_beta} we have to impose $j=m\cdot2^i\leq a_{n-i-l}$ for $n\geq3$; since the situation $m=i=1$ has already been dealt with in the previous part (iii), it is sufficient to consider the case $n=l+i+3$, $4\leq m\cdot 2^i=j\leq a_{l+2}$, where $l\geq0$, $i\geq1$. In order to compute the $f$--invariants, we calculate the effect of the map \eqref{elliptic_beta_map} w.r.t.\ $\left(2^{i+2},z_{i,m}\right)$: Since 
\begin{equation}\label{mod_higher_power}
\begin{split}
v_1^{6\cdot2^l}&=z_{i,m}v_1^{6\cdot 2^l-j}+2jv_1^{6\cdot2^l-3}v_2\\
v_1^{9\cdot2^l}&=z_{i,m}\left(v_1^{9\cdot2^l-j}+2jv_1^{9\cdot2^l-j-3}v_2\right)+4j^2v_1^{9\cdot2^l-6}v_2^2
\end{split}
\end{equation}
we calculate for $l\geq0$, $i\geq1$, and odd $s\geq1$:
\begin{align*}
x_{l+i+3}^s&\equiv v_2^{s\cdot2^{l+i+3}}+2^{i+1}v_1^{3\cdot2^l}v_2^{\left(s2^{i+3}-1\right)2^l}+\\
&\quad+3s\cdot2^iv_1^{6\cdot2^l}v_2^{\left(s2^{i+3}-2\right)2^l}& &\mod\left(2^{i+2}, v_1^{9\cdot2^l}\right)\\
&\equiv v_2^{s\cdot2^{l+i+3}}+2^{i+1}v_1^{3\cdot2^l}v_2^{\left(s2^{i+3}-1\right)2^l}& &\mod\left(2^{i+2},z_{i,m}\right)
\end{align*}
hence
\begin{align*}
x^s_{l+i+3}&\equiv E_3^{s\cdot2^{l+i+3}}+2^{i+1}E_3^{\left(s\cdot2^{i+3}-1\right)2^l}\mod2^{i+2}D^{\Gamma}_{24s\cdot2^{l+i}}+z_{i,m}\cdot M^{\Gamma}_{24s\cdot2^{l+i}-j}
\end{align*}
and due to \eqref{mod_higher_power}, application of Lemma \ref{E_3_congruences} and Lemma \ref{removal_of_second_summand} yields the claim.
\end{proof}


\bibliography{refbib_edited}

\providecommand{\bysame}{\leavevmode\hbox to3em{\hrulefill}\thinspace}
\begin{thebibliography}{MRW77}

\bibitem[Beh09]{Behrens:2009dz}
Mark Behrens, \emph{Congruences between modular forms given by the divided
  {$\beta$} family in homotopy theory}, Geom. Topol. \textbf{13} (2009), no.~1,
  319--357.

\bibitem[BL08]{Behrens:2008fp}
Mark Behrens and Gerd Laures, \emph{Beta-family congruences and the
  f-invariant}, arXiv:0809.1125 [math.AT], 2008.

\bibitem[HN07]{Hornbostel:2007ss}
Jens Hornbostel and Niko Naumann, \emph{Beta-elements and divided congruences},
  Amer. J. Math. \textbf{129} (2007), no.~5, 1377--1402.

\bibitem[Lau99]{Laures:1999sh}
Gerd Laures, \emph{The topological {$q$}-expansion principle}, Topology
  \textbf{38} (1999), no.~2, 387--425.

\bibitem[MRW77]{Miller:1977ya}
Haynes~R. Miller, Douglas~C. Ravenel, and W.~Stephen Wilson, \emph{Periodic
  phenomena in the {A}dams-{N}ovikov spectral sequence}, Ann. Math. (2)
  \textbf{106} (1977), no.~3, 469--516.

\bibitem[Rav04]{Ravenel:2004xh}
Douglas~C. Ravenel, \emph{{Complex cobordism and stable homotopy groups of
  spheres. 2nd ed.}}, {Providence, RI: AMS Chelsea Publishing.}, 2004.

\bibitem[Shi81]{Shimomura:1981oy}
Katsumi Shimomura, \emph{Novikov's {${\rm Ext}\sp{2}$} at the prime {$2$}},
  Hiroshima Math. J. \textbf{11} (1981), no.~3, 499--513.

\bibitem[vB08]{Bodecker:2008pi}
Hanno von Bodecker, \emph{On the geometry of the $f$-invariant}, Ph.D. thesis,
  Ruhr-{U}niversit{\"a}t {B}ochum, 2008.

\bibitem[vB09]{Bodecker:2009kx}
Hanno von Bodecker, \emph{On the f-invariant of products}, arXiv:0909.3968
  [math.AT], 2009.

\end{thebibliography}
\end{document}